\documentclass{amsart}
\usepackage{graphicx}
\usepackage{amssymb,amscd,amsthm,amsxtra}
\usepackage{latexsym}
\usepackage{epsfig}

\newtheorem{thm}{Theorem}[section]

\newtheorem{lem}[thm]{Lemma}
\newtheorem{prop}[thm]{Proposition}
\theoremstyle{definition}
\newtheorem{defn}[thm]{Definition}
\theoremstyle{remark}
\newtheorem{rem}[thm]{Remark}
\numberwithin{equation}{section}

\newcommand{\be}{\begin{equation}}
\newcommand{\ee}{\end{equation}}

\newcommand{\R}{\mathbb R}

\newcommand{\eps}{\varepsilon}

\newcommand{\p}{\partial}

\newcommand{\comment}[1]{}

\begin{document}

\title[Nonlocal phase transitions]{Rigidity of minimizers in nonlocal phase transitions} %
\author{O. Savin}
\address{Department of Mathematics, Columbia University, New York, NY 10027}
\email{\tt  savin@math.columbia.edu}

\begin{abstract}

We obtain the classification of certain global bounded solutions for semilinear nonlocal equations of the type
$$\triangle^s u=W'(u) \quad \mbox{in} \quad \R^n, \quad \quad \mbox{with} \quad s \in (1/2 ,1),$$
where $W$ is a double well potential. 
\end{abstract}

\maketitle

\section{Introduction}

In this article we extend to the case of the fractional Laplacian $\triangle ^s$ with $s \in (1/2 ,1)$ the results from \cite{S1}, \cite{S2} concerning a conjecture of De Giorgi about the classification of certain global bounded solutions for semilinear equations of the type
$$\triangle u=W'(u),$$
where $W$ is a double well potential. 

We consider the Ginzburg-Landau energy functional with nonlocal interactions
$$J(u, \Omega)=\frac14 \int_{\R^n \times \R^n \setminus ( \mathcal C \Omega \times \mathcal C \Omega)} 
 \frac{ (u(x)-u(y))^2}{|x-y|^{n+2s}} \,dxdy+ \int_\Omega W(u) \, dx, $$ with $|u|\le 1$. Here $W$ is a double-well
potential with minima at $1$ and $-1$ satisfying
$$W \in C^2([-1,1]), \quad W(-1)=W(1)=0, \quad \mbox{$W>0$ on $(-1,1)$},$$
$$W'(-1)=W'(1)=0, \quad W''(-1)>0, \quad W''(1)>0.$$
The
classical double-well potential $W$ to have in mind is
$$W(s)=\frac{1}{4}(1-s^2)^2.$$

Physically $u \equiv- 1$ and $u \equiv 1$ represent the stable ``phases". A critical function for the energy $J$ corresponds to a phase transition with nonlocal interaction between these states, and it satisfies the Euler-Lagrange equation
$$\triangle ^s u=W'(u),$$
where $\triangle^s u$ is defined as
$$\triangle ^s u(x)= PV \int_{\R ^n} \frac{u(y)-u(x)}{|y-x|^{n+2s}} \, dy . $$ 

Our main result provides the classification of minimizers with asymptotically flat level sets.

\begin{thm} \label{TM}
Let $u$ be a global minimizer of $J$ in $\mathbb{R}^n$ with $s \in (\frac 12 ,1)$. If the $0$ level set $\{u=0\}$ is asymptotically flat at $\infty$,
then $u$ is one-dimensional.
\end{thm}

A more quantitative version of Theorem \ref{TM} is given in Theorem \ref{c1alpha}. 

In a subsequent work we will treat also the case $s=\frac 12$ which requires some modifications of the methods presented in this paper. We remark that Theorem \ref{TM} when $s \in (0, \frac 12)$ was obtained recently by Dipierro, Valdinoci and Serra \cite{DVS}.  

It is known that blowdowns of the level set $\{u=0\}$ have different behavior depending on the value of $s$.
 If $ s\ge 1/2$, there are sequences $\eps_k \{u=0\}$ with $\eps_k \to 0$ that converge uniformly on compact sets to a minimal surface and, if $s< 1/2$ they converge to a $s$-nonlocal minimal surface.
 This follows from a $\Gamma$-convergence result together with a uniform density estimate of level sets of minimizers which were obtained by the author and Valdinoci in \cite{SV1}, \cite{SV2}, see for example Corollary 1.7 in \cite{SV1}.

From the classification of global minimal surfaces in low dimensions we find that the level sets of minimizers of $J$ are always asymptotically flat at $\infty$ in dimension $n \le 7$ if $ s\ge 1/2$, and we obtain the following corollary of Theorem \ref{TM}.

\begin{thm}\label{7min}
A global minimizer of $J$ is one-dimensional in dimension $n \le 7$ if  $s \in (\frac 12 ,1)$.
\end{thm}

Another consequence of Theorem \ref{TM} is the following version of De Giorgi's conjecture to the fractional Laplace case.

\begin{thm}{\label{8min}}
Let $u \in C^2(\mathbb{R}^n)$ be a solution of
\begin{equation}{\label{8min1}}
\triangle^s u = W'(u),
\end{equation}
with $ s\in (1/2,1)$, such that
\begin{equation}{\label{8min2}}
|u| \le 1, \quad \partial_n u>0, \quad \lim_{x_n \to \pm
\infty}u(x',x_n) =\pm 1.
\end{equation}

Then $u$ is one-dimensional if $n \le 8$.

\end{thm}

Theorem \ref{7min} and Theorem \ref{8min} without the limit assumption in \eqref{8min2} have been proved in 2 and 3 dimensions using stability inequality methods. In dimension $ n=3$ and for $s \ge 1/2$  they have been established by Cabre and Cinti \cite{CC2},  and in dimension $n=2$ for all $s \in (0,1)$ by Sire and Valdinoci \cite{SiV}, see also \cite{CC1}, \cite{CS2}, \cite{CSo}.

It is not difficult to show that the $\pm 1$ limit assumption implies that $u$ is a global minimizer in $\mathbb{R}^n$, see for example Theorem 1 in \cite{PSV}. Since $\{u=0\}$ is a graph, it is asymptotically flat in dimension $n \le 8$ and Theorem \ref{TM} applies.

Similarly we see that if
the $0$ level set is a graph in the $x_n$ direction
 which has a one sided linear bound at
  $\infty$ then the conclusion is true in any dimension.

\begin{thm}
 If $u$ satisfies (\ref{8min1}), (\ref{8min2}), and $$\{u=0\} \subset \{x_n < C(1+|x'|)\},$$ and $ s\in (\frac 12 ,1)$ then $u$ is one-dimensional.
\end{thm}

Our proof of Theorem \ref{TM} follows closely the one for the classical Laplacian given in \cite{S2}. The main steps consist in 1) finding some appropriate families of radial subsolutions, 2) applying a version of weak Harnack inequality and 3) a $\Gamma$-convergence result. Some new technicalities are present in our setting due to the nonlocal nature of the equation. For example in the improvement of flatness property Theorem \ref{c1alpha}, we need to impose a geometric restriction to the level set $\{u=0\}$ possibly outside the flat cylinder $\mathcal C(l,\theta)$.

We prove Theorem \ref{TM} by making use of the extension property of the fractional Laplacian of Caffarelli-Silvestre \cite{CS}. Precisely we consider the extension $U(x,y)$ of $u(x)$ in $\R^{n+1}_+$ such that
$$div(y^{a}\nabla U)=0 \quad \mbox{in $\R^{n+1}_+$,} \quad \quad U(x,0)=u(x), \quad \quad a:=1-2s \in (-1,1),$$
and then $$\triangle ^s u(x)= c_{n,s} \lim_{y \to 0^+} y^a U_y(x,y),$$
with $c_{n,s}$ a constant that depends only on $n$ and $s$.
Then global minimizers of $J(u)$ in $\R^n$ with $|u| \le 1$ correspond to global minimizers of the ``extension energy" $\mathcal J(U)$ with $|U| \le 1$ where
$$ \mathcal J(U):= \frac{c_{n,s}}{2} \int |\nabla U|^2 y^a \, dx dy + \int W(u) dx.$$ 
After dividing by a constant and relabeling $W$ we may fix $c_{n,s}$ to be $1$. 
We obtain an improvment of flatness property for the level sets of minimizers of 
$\mathcal J$ which are defined in large balls $\mathcal B_R^+$, see Theorem \ref{c1alpha}. We remark that the principal use of the 
extension is to make the various subsolution computations easier to handle and it is not essential to the method of proof. 

The paper is organized as follows. In Sections 2 and 3 we introduce some notation and then construct a family of axial subsolutions. In Section 4 we provide certain ``viscosity solution" properties of the level set $\{u=0\}$. In Section 5 we obtain 
a Harnack inequality of the $0$ level set and in Section 6 we prove Theorem \ref{c1alpha}.  

\section{Notation and preliminaries}

We introduce the following notation:

\noindent
We denote points in $\R^n$ as $x=(x',x_n)$ with $x' \in \R^{n-1}$. The ball of center $z$ and radius $r$ is denoted by $B_r(z)$,
$$B_r(z):=\{x\in \R^n| |x-z| < r\}, \quad \quad B_r:=B_r(0).$$
The cylinder with base $l$ and height $\theta$ is denoted by $\mathcal C(l,\theta) \subset \R^n$
$$\mathcal C(l,\theta) := \{x|  \quad |x'| \le l, \quad |x_n| \le \theta \}.$$
Points in the extension variables $\R^{n+1}_+$ are denoted by $(x,y)$ with $y>0$, and the ball of radius $r$ as $\mathcal B_r^+$
$$\mathcal B^+_r :=\{(x,y) \in \R^{n+1}_+| \quad |(x,y)| <r \} \quad  \subset \R^{n+1}.$$
Given a function $U(x,y)$ we define $u$ its trace on $\{y=0\}$
$$u(x)=U(x,0).$$
Also let
$$a:=1-2s \in (-1,0),$$
and 
$$\triangle_a U := \triangle U + a\frac{U_y}{y}= y^{-a} div(y^a \nabla U),$$
$$\p_{y}^{1-a}U (x):=  \lim_{y \to 0^+} y^a U_y(x,y)=\frac {1} {1-a} \, \lim_{y \to 0^+} \, \, y^{a-1} \,  \left( U(x,y)-U(x,0)\right).$$
We define the energy $\mathcal J$ as 
$$ \mathcal J(U, \mathcal B_R^+):= \frac 12 \int_{\mathcal B^+_R} |\nabla U|^2 y^a \, dx dy + \int_{B_r} W(u) dx,$$
and a critical function $U$ for $\mathcal J$ satisfies the Euler-Lagrange equation
\be\label{eq}
\triangle_a U=0, \quad \quad \p_{y}^{1-a}U=W'(u).
\ee

In \cite{PSV} Theorem 2, see also \cite{CS1}, it was proved the existence and uniqueness up to translations of a global minimizer of $\mathcal J$ in 2D which is increasing in the first variable and which has limits $\pm 1$ at infinity. Precisely there exists a unique $G:\R^2_+ \to (-1,1)$ that solves the equation \eqref{eq} such that $G(t,y)$ is increasing in the $t$ variable and its trace $g(t):=G(t,0)$ satisfies
$$g(0)=0 , \quad \lim_{t \to \pm \infty} g(t)= \pm 1.$$
Moreover, $g$ and $g'$ have the following asymptotic behavior 
$$1-|g| \sim \min\{ 1, |t|^{-2s}\}, \quad \quad g' \sim \min \{ 1, |t|^{-1-2s} \},$$
and since $a \in (-1,0)$ we have $\mathcal J(G, \R^2_+) < \infty.$

Since $\triangle_a G_t=0$ and $G_t \ge 0$, we easily conclude that
\be\label{NG}
|\nabla G| \le C \min\{1,r^{-1}\}, \quad \quad G_t \ge c \, r^{-1-2s}.
\ee
where $r$ denotes the distance to the origin in the $(t,y)$-plane.

In Theorem \ref{c1alpha} we show that the only global minimizer of $\mathcal J$ that has asymptotically flat level sets on $y=0$ is $G(x_n,y)$ up to translations and rotations.
\

For simplicity of notation we assume that $W$ is uniformly convex outside the interval $[g(-1),g(1)]$.

Constants that depend on $n$, $s$, $W$, $G$ are called universal constants, and we denote them by $C$, $c$. 
In the course of the proofs the values of $C$, $c$ may change from line to line when there is no possibility of confusion. If the constants depend on other parameters, say $\theta$, $\rho$, then we denote them by $C(\theta, \rho)$ etc. 

\section {2D barriers}

In this section we construct two families of comparison functions $G_R$ and $\Psi_R$ which are perturbations of the solution $G$.

\begin{lem}[Radial subsolutions]\label{rs1}
For all large $R$, there exist continuous functions $G_R:\R^2 \to (-1,1]$, and $\delta>0$ small, $C$ large universal constants such that:

\ 

1) $G_R=1$ outside $\mathcal B_{R^{1-\delta}}^+ \cup \left ( (-\infty,0] \times [0,R^{1-\delta}] \right )$, 

\

2) $G_R(t,y)$ is nondecreasing in $t$, and $\p_t G_R=0$ outside $\mathcal B_{R^{1-\delta}}^+$,

\

3) $$|G_R-G| \le \frac {C}{R} \quad \mbox{in $ \mathcal B_4^+$},$$

\

4)
$$\triangle_a G_R+ \frac{2(n-1)}{R}\,  | \nabla G_R| \le 0, $$
and on $y=0$: 
$$\partial_y^{1-a} G_R < W'(G_R) \quad \quad \mbox{if $t \notin [-1,1]$.}  $$

\end{lem}

The inequalities in 4) are understood in the viscosity sense.

Notice that by \eqref{NG}, property 3) implies that $$G_R(t,y) \le G \left(t+ \frac{C'}{R},y \right) \quad \mbox{ in}  \quad \mathcal B_4^+.$$

We remark that property 3) and the inequality above hold in any ball $\mathcal B^+_K$, for a fixed large constant $K$, provided that we replace $C/R$, $C'/R$ by $C(K)/R$, $C'(K)/R$.
\begin{proof}
We begin with the following claim whose proof we provide at the end.

{\it Claim:} For each $\alpha \in (1,1-a)$ there exists $H$ a homogenous of degree $\alpha$ function such that  
$$H\ge r^\alpha, \quad \triangle _a H \le - r^ {\alpha -2}, \quad |\nabla H| \le C r^{\alpha-1},
 \quad \partial_y^{1-a} H \le C |t|^{\alpha - (1-a)}.$$ 
Here $r$ denotes the distance to the origin and $C=C(\alpha)$ depends on the universal constants and $\alpha$. 

\

Fix such an $\alpha$ and define 
\be\label{HR}
H_R:= \min \left \{ G+ \frac {C_0}{R}( H + C_1), \quad 1 \right \},
\ee
with $C_0$, $C_1$ large constants to be specified later. 

We define $G_R$ as the infimum over all left translations of $H_R$ i.e.
$$G_R(t,y) = \inf_{l \ge 0} H_R(t+l,y).$$

Since $|G|<1$ we have $H_R > -1$, and $H_R=1$ outside $\mathcal B^+_{R^{1-\delta}}$ 
provided that $\delta$ is chosen sufficiently small such that $(1-\delta) \alpha >1$. 
Properties 1) and 2) are clearly satisfied.

Notice that $H$ is increasing in a band $[C, \infty) \times [0,4]$ and we obtain that $H_R$ is increasing in $[-4,\infty) \times [0,4]$. This gives $G_R=H_R$ in $\mathcal B_4^+$ and property 3) is satisfied.

The properties of $H$ and \eqref{NG} imply that in the set $\{H_R<1\}$ we have
$$|\nabla H_R| \le C \min \{1,r^{-1}\} + C C_0 R^{-1} r^{\alpha -1},$$
and $$\triangle _a H_R \le - C_0 R^{-1} r^{\alpha-2}.$$
Then the first inequality in 4) holds for $H_R$ provided that $C_0$ is chosen sufficiently large, and therefore holds also for $G_R$ as the infimum over translations of $H_R$. 

 On $y=0$ in the set $\{H_R<1\}$ we have
$$\p_y^{1-a} H_R=\p_y^{1-a} G + C_0R^{-1} \p_y^{1-a} H \le W'(G) + C R^{-1} |t|^{\alpha -(1-a)}.$$

From the behavior of $g$ and $g'$ for large $t$, 
we see that the minimum of $H_R(t,0)$ occurs at some $t=q_R \sim - R^{1/(2s+\alpha)} \ll -1$ and
$$\|(H_R-G)(t,0)\|_{L^\infty([q_R, \infty)} \to 0 \quad \mbox{ as} \quad  R \to \infty.$$ 
Since $W'' \ge c$ outside $[g(-1), g(1)]$ we find that when $ t \in [q_R,\infty) \setminus[-1,1] $ and $\{H_R<1\}$ we have 
$$W'(H_R)-W'(G) \ge \frac c 2 (H_R-G) \ge c' R^{-1}(|t|^\alpha +  C_1),$$
thus, if $C_1$ is sufficiently large,
$$\p_y^{1-a} H_R < W'(H_R) \quad \quad  \mbox{in} \quad [q_R, \infty) \setminus [-1,1].$$
Now the second inequality of 4) is satisfied by $G_R$ as the infimum of left translations of $H_R$. 

\

{\it Proof of Claim:} We find $H$ as a perturbation of the function $ C y ^\alpha$ near $y=0$. 
Notice that $y^{1-a}$ is $\triangle_a$-harmonic, thus $y^\alpha$ is $\triangle_a$-superharmonic for $\alpha < 1-a$. 
However $C y^\alpha$ does not satisfy the first and the third property given in the claim. 
 
We write $H$ in polar coordinates as, $H=r^\alpha h(\theta)$ with $h$ an even function with respect to $\pi/2$ and then
\be\label{ah}
r^{2-\alpha}\triangle_a H=h'' + \alpha(\alpha +a) h + a \cot \theta \, \,  h' ,
\ee
\be \label{yh}
\p_y^{1-a} H= r^{\alpha-(1-a)}  \, \partial_\theta ^{1-a} h. 
  \ee
For all small $\sigma$, the function 
$$h_ \sigma=\sigma + \theta^{1-a} -  \theta^2,$$ 
gives a negative right hand side in \eqref{ah} when $\theta$ belongs to a small fixed interval $[0,c]$. 
We choose first $M$ large and then $\sigma$ small such that 
the graphs of $M h_\sigma$ and $(\sin \theta)^\alpha$ become tangent by above at some point in the interval $[0,c]$.
 Now we ``glue" parts of the two graphs in a single graph of a $C^{1,1}$ function $\tilde h$. 
Now it is easy to check that all properties hold by taking $h$ a large multiple of $\tilde h$.

\end{proof}

From the construction of $H_R$, $G_R$ we see that both of them decrease with $R$ as we increase $R$.

Next we construct a similar family $\Psi_R$ which can be compared with $G_{\bar R}$ even when $R$ and $\bar R$ have different orders of magnitude.
  
\begin{lem}\label{rs2}
There exist functions $G_R$ and $\Psi_R$ that satisfy the properties 1)-4) of Lemma \ref{rs1} for some $\delta$, $C$ universal 
such that
$$G_R(t + R^{-\sigma},y) \ge \Psi_{R^{1-\sigma}}(t,y),$$
with $\sigma \in (0, \delta/3)$ small universal.
\end{lem}

\begin{proof}
Denote by $G_{R, \alpha}$ the function constructed in Lemma \ref{rs1}. 

We choose $G_R:=G_{R,\alpha}$, $\Psi_R:= G_{R, \beta}$
 for some fixed $\alpha$, $\beta$ such that $1< \beta <\alpha < 1-a$. We take $\delta=\min \{\delta(\alpha), \delta(\beta) \}$, and $C= \max\{ C(\alpha), C(\beta) \}$ and then Lemma \ref{rs1} holds for both $G_R$ and $\Psi_R$ with the same constants $\delta$ and $C$.

We show that
$$H_{R, \alpha}(t+ R^{-\sigma},y) \ge H_{R^{1-\sigma}, \beta}(t,y),$$
with $H_{R,\alpha}$ defined as in \eqref{HR}, and the lemma follows by taking the infimum over the left translations.

In the inequality above it suffices to restrict to the set where $\{H_{R,\alpha} <1.\}$ We have
$$H_R \ge G + R^{-1}(c_1 r^\alpha + c_2),$$
for some constants $c_1$, $c_2 $ depending on $\alpha$. After a translation of $R^{-\sigma}$ we obtain (see \eqref{NG})
$$H_R(t+ R^{-\sigma},y) \ge G(t,y) +c R^{-\sigma} \min \{1, r^{-1-2s}\} + \frac 12 R^{-1}(c_1 r^ \alpha + c_2).$$
When $r\ge1$ we use the inequality $a+ b \ge a ^\mu b^{1-\mu}$ for $\mu>0$ small, and we find
\be\label{2d1}
H_R(t+ R^{-\sigma},y) \ge G(t,y) + c(\alpha) R^{-\eta}(r^{\gamma}+1),
\ee
with
$$\gamma=\alpha(1-\mu)-\mu(1+2s), \quad \eta= 1-\mu + \sigma \mu,$$ (and $\eta > \sigma$.) We choose $\mu$ small and then $\sigma$ such that
$\gamma > \beta$ and $\eta < 1-\sigma$. Then the right hand side of \eqref{2d1} is grater than
$$G + R^{\sigma -1} (C_1(\beta)r^ \beta + C_2(\beta)) \ge H_{R^{1-\sigma}, \beta},$$
for all large $R$, and the lemma is proved.

\end{proof}

\begin{rem}\label{r00}
Using the monotonicity of $\Psi_r$ with respect to $r$, we have
$$ G_R(s + R^{-\sigma},y) \ge \Psi_{r}(s,y), \quad \quad \forall \, r \ge R^{1-\sigma}.$$
\end{rem}

\section{Estimates for $\{u=0\}$}

In this section we derive properties of the level sets of solutions to
\begin{equation}{\label{2min1}}
\triangle_a U = 0, \quad \quad \p_y^{1-a} U=W'(U),
\end{equation}  
which are defined in large domains.

In the next lemma we find axial approximations to the 2D solution $G$.

\begin{lem}[Axial approximations]\label{l1}
Let $G_R: \R^2_+ \to (-1,1]$ be the function constructed in Lemma \ref{rs2}. Then its axial rotation in $\R^{n+1}$
$$\Phi_R (x,y):=G_R(|x|-R,y)$$
satisfies 

1) $\Phi_R=1$ outside $\mathcal B_{R+R^{1-\delta}}^+$, 

\

2) 
$$\triangle_a \Phi_R \le 0 \quad \mbox{ in} \quad \R^{n+1}_+,$$ and $$\p_y^{1-a} \Phi_R < W'(\Phi_R) \quad \mbox{when} \quad |x|-R \notin  [-1,1] .$$ 
\end{lem}

Let $\phi_R(x)=\Phi_R(x,0)$ denote the trace of $\Phi_R$ on $\{y=0\}$. 
Notice that $\phi_R$ is radially increasing, and $\{\phi_R=0\}$ is a sphere which is in a $C/R$-neighborhood of the sphere of radius $R$.

\begin{proof}

We have
$$\triangle _a \Phi_R(x,y)= \triangle _a G_R(s,y)  + \frac{n-1}{R+s}\, \partial_s \, G_R(s,y), \quad \quad s=|x|-R,$$
$$\p_y^{1-a} \Phi_R(x,0)=\p_y^{1-a} G_R(s,0).$$
The conclusion follows from Lemma \ref{rs2} since $\partial_s G_R=0$ 
when $|s| \ge R^{1-\delta}$ and $R+s>R/2$ when $|s| < R^{1-\delta}$.

\end{proof}

\begin{defn}\label{d1}
 We denote by $\Phi_{R,z}$ the translation of $\Phi_R$ by $z$ i.e.
$$\Phi_{R,z}(x,y):= \Phi_R(x-z,y)=G_R(|x-z|-R,y).$$
Similarly we define $\Psi_{R,z}$ the axial rotation of the other 2D solution $\Psi_R$ given in Lemma \ref{rs2},
$$\Psi_{R,z}(x,y):=\Psi_R(|x-z|-R,y).$$
Clearly $\Psi_{R,0}$ satisfies properties 1), 2) of Lemma \ref{l1}. 

\end{defn}

{\it Sliding the graph of $\Phi_R$:}

Assume that $u$ is less than $\phi_{R,x_0}$ in $B_{2R}(x_0)$. 
By the maximum principle we obtain that $U< \Phi_{R,z}$ with $z=x_0$ 
in $\mathcal B_{2R}(x_0,0)$ (and therefore globally.)
We translate the function $\Phi_R$ above by moving continuously the center $z$, 
and let's assume that it touches $U$ by above, say for simplicity when $z=0$, 
i.e. the strict inequality becomes equality for some contact point $(x^*,y^*)$. 
From Lemma \ref{l1} we know that $\Phi_R$ is a strict supersolution away from $\{y=0\}$, 
and moreover the contact point must satisfy $y^*=0$, $|x^*| -R \in [-1,1]$, 
that is it belongs to the annular region $B_{R+1}\setminus B_{R-1}$ in the $n$-dimensional subspace $\{y=0\}$.

\begin{lem}[Estimates near a contact point] \label{l2}
Assume that the graph of $\Phi_R$ touches by above the graph of $U$ at a point $(x^*,0, u(x^*))$ with $x^* \in B_{R+1} \setminus B_{R-1}$. Let $\pi(x^*)$ be the projection of $x^*$ onto the sphere $\p B_R$. Then in $\mathcal B_1(\pi(x^*),0)$

1) $\{u=0\}$ is a smooth hypersurface in $\R^n$ with curvatures bounded by $\frac C R$ which stays in a $\frac C R$ neighborhood of $\p B_R$.  

2) $$|U-G(x \cdot \nu -R,y)| \le \frac CR, \quad \quad \nu:= \pi(x^*)/R.$$

\end{lem}

\begin{proof}
Assume for simplicity that $x^*$ is on the positive $x_n$ axis and therefore $\pi(x^*)=Re_n$, $|x^*-Re_n| \le 1$. By Lemma \ref{l1} we have
$$U \le \Phi_R \le G \left(|x|-R + \frac C R, y \right) \le G \left (x_n-R + \frac{C'}{R},y \right)=:V \quad \quad \mbox{in} \quad \mathcal B_{3}(R e_n).$$
Both $U$ and $V$ solve the same equation \eqref{2min1}, and $$(V-U)(x^*,0) \le \frac {C''}{ R}.$$ Since $V-U \ge 0$ satisfies
$$\triangle_a (V-U)=0, \quad \quad \p_y^{1-a} (V-U)= b(x)(V-U), $$
$$ b(x):= \int_0^1 W''(tu(x) + (1-t)v(x)) dt,$$
we obtain $$|V-U| \le \frac C R \quad \quad \mbox{in} \quad \mathcal B_{5/2}(R e_n),$$
from the Harnack inequality with Neumann condition for $\triangle_a$. Moreover since $b$ has bounded Lipschitz norm and $s> 1/2$ we obtain that $U-V \in C_x^{2,\alpha}$ for some $\alpha>0$, and 
$$\|U-V\|_{C_x^{2,\alpha}(\mathcal B_2(Re_n))} \le \frac C R,$$
by local Schauder estimates. This easily implies the lemma.

\end{proof}

\begin{rem}\label{r1}
If instead of $\mathcal B_1((\pi(x^*),0))$ we write the conclusion in $\mathcal B_K((\pi(x^*),0))$ for some large, fixed constant $K$, then we need to replace $\frac C R$ by $\frac{C(K)}{R}$. Here $C(K)$ represents a constant which depends also on $K$.
\end{rem}

Next we obtain estimates near a point on $\{u=0\}$ which admits a one-sided tangent ball of large radius $R$.

\begin{lem}\label{l3}
Assume that $U$ is defined in $\mathcal B_{2R}^+$, satisfies \eqref{2min1}, and that

a) $B_R(-Re_n) \subset \{u<0\}$ is tangent to $\{u=0\}$ at $0$,

b) there is $x_0 \in B_{R/2}(-Re_n)$ such that $u(x_0) \le -1+c$ for some $c>0$ small. 

\noindent
Then 

1)  $\{u=0\}$ is smooth in $B_1$ and has curvatures bounded by $\frac C R$.

2) $|U-G(x_n,y)| \le \frac C R$ in $\mathcal B_1$.

\end{lem}

\begin{proof}
Assume first that $u< \phi_{R/8,z}$  for $z =-Re_n$.
 
We translate the graph of $\Phi_{R/8,z}$ by moving $z$ continuously upward on the $x_n$ axis. We stop when the translating graph becomes 
tangent by above to the graph of $U$ for the first time. 
Denote by $(x^*,0,u(x^*))$ the contact point and by $z^*$ the final center $z$ and by $\pi(x^*)$ the projection of $x^*$ onto $\p 
B_{R/8}(z^*)$. 

By Lemma \ref{l2}, $\{u=0\}$ must be in a $\frac{C_1}{R}$ neighborhood of $\p B_{R/8}(z^*) \cap B_1(\pi( x^*))$ for some $C_0$ universal. This implies
$$z^* = t e_n \quad\mbox{ with} \quad  \quad t \in  \left[- \frac R  8 - \frac{C_1}{ R} ,- \frac R  8 + \frac{C_1}{R}\right].$$    
Moreover, $\pi (x^*) \in B_{C_2}$ since otherwise $\pi (x^*)$ is at a distance greater than $\frac{1}{R} \frac{C_2^2}{8} > \frac{C_1}{R}$ in the interior of the ball $B_{R}(-R e_n)$, hence $\{u=0\}$ must intersect this ball and we reach a contradiction. 

Now we apply Lemma \ref{l2} and Remark \ref{r1} at $\pi(x^*)$ and obtain 
the conclusion of the lemma.

It remains to show that $u < \phi_{R/8,-R e_n}$. By hypothesis b) and Harnack inequality we see that $u$ is still 
sufficiently close to $-1$ in a whole ball $B_{R_0}(x_0)$ for some large universal $R_0$, 
and therefore $u<\phi_{R_0/2,x_0}$ provided that $c$ is sufficiently small. 
Now we deform $\Phi_{R_0/2,x_0}$ by a continuous family of functions $\Phi_{r,z}$ and first 
we move $z$ continuously from $x_0$ to $-R e_n$ and then we increase the radius $r$ from $R_0$ to $R/8$.
By Lemma \ref{l2}, the graphs of these functions cannot touch the graph of $U$ by above and we obtain the desired inequality.
With this the lemma is proved. 
    
\end{proof}

In the next lemma we prove a localized version of Lemma \ref{l3}. 

\begin{prop}\label{p1}
Assume that $U$ satisfies the equation in $\mathcal B_{R^{1-\sigma}}$ with $\sigma$ small, universal as in Lemma \ref{rs2}, and 

a) $B_R(-Re_n) \cap B_{R^{\frac 12-\sigma}} \subset \{u<0\}$ is tangent to $\{u=0\}$ at $0$,

b) all balls of radius $\frac 14 R^{1-\sigma}$ which are tangent by below to $\p B_R(-Re_n)$ in $B_{R^{\frac 12-\sigma}}$ are included in $\{u<0\}$,

c) there is $x_0 \in B_{R^{1-\sigma}/4}(-\frac 12 R^{1-\sigma} e_n)$ such that $u(x_0) \le -1+c$. 

\noindent
Then in $B_1$ we have that $\{u=0\}$ is smooth and has curvatures bounded by $\frac C R$.

\end{prop}

\begin{proof}

As in Lemma \ref{l3}, we slide the graph of $\Phi_{R/8,z}$ in the $e_n$ direction till it touches the graph of $U$, except that now we restrict only to the region 
\be\label{CR}
\mathcal C_R:=\left \{ |x'| \le \frac 12 R^{\frac 12-\sigma}, \quad |x_n| \le \frac 12 R^{1-\sigma}, \quad |y| \le \frac 12 R^{1-\sigma} \right\}.
\ee
In order to repeat the argument above we need to show that the first contact point is an interior point and it occurs in $\mathcal C_{R/2}$. For this it suffices to prove that
\be\label{uphi}
U <  \Phi_{R/8,z_0} \quad \quad \mbox{in} \quad \mathcal C_R \setminus \mathcal C_{ R /2}, \quad \quad z_0:=\left(-\frac R8  + \frac{C_1}{R} \right)e_n.
\ee

We estimate $U$ by using the functions $\Psi_{R,z}$ given in Definition \ref{d1}.
 Notice that Lemma \ref{l2} holds if we replace $\Phi_R$ by $\Psi_R$.

Now we slide the graphs $\Psi_{r, z}$ with $r:=\frac 14 R^{1-\sigma}$ and $|z'|\le R^{\frac 12 -\sigma}$, $z_n=-2r$ upward in the $e_n$ direction. 
We use hypotheses b), c) and as in the proof of  Lemma \ref{l3} we find
$\Psi_{r, z} > U$ as long as $B_r(z)$ is at distance greater than $Cr^{-1}$ from $\p B_{R}(-R e_n)$.
We obtain that 
\be\label{urho}
U(x) <  \Psi_r (d_1(x) + C r^{-1},y), 
\ee
 where $d_1(x)$ is the signed distance to $\p B_R(-Re_n)$. From Remark \ref{r00} we have
$$ \Psi_r(s,y) \le  G_{R/8}(s+(R/8)^{-3\sigma},y).$$

We obtain
\be\label{ug}
U(x,y) < G_{R/8}(d_1(x) + 2 R^{-3\sigma},y).
\ee
Let $d_2(x)$ represent the distance to $\p B_{R/8}(z_0)$. Then in the region $\mathcal C_R \setminus \mathcal C_{R/2}$ we have either 

a) $|x'| \ge \frac 12  (R/2)^{\frac 12-\sigma}$ and then
\be\label{d12}
d_2(x)-d_1(x) \ge -\frac{C_1}{R} + \frac {1}{R}|x'|^2 \ge 2 R^{-3\sigma},   
\ee
or 

b) $\min \{|x_n|, y\} \ge R^{1-\sigma}/8$ and then both $(d_2(x),y)$ and $(d_1(x)+ 2R^{-\sigma},y)$ are outside $B^+_{1-\delta} \subset \R^2$, thus $G_{R/8}$ at these two points has the same value.

From \eqref{ug} we find
\be\label{ug1}
U(x,y) < G_{R/8}(d_2(x),y) \quad \mbox{in} \quad \mathcal C_R \setminus \mathcal C_{R/2},
\ee
and \eqref{uphi} is proved.

\end{proof}

Next we consider the case in which the $0$ level set of $u$ is tangent by above at the origin to the graph of a quadratic polynomial. 

\begin{prop}\label{p2}
Let $U$ satisfies the equation in $\mathcal B_{R^{1-\sigma}}$ and the hypothesis c) of Proposition \ref{p1}. Assume the surface 
$$\Gamma:=\left \{x_n= \sum_1^{n-1} \frac{a_i}{2} x_i^2 + b' \cdot x'  \right \} \cap B_{R^{\frac 12-\sigma}} \quad \mbox{ with} \quad |b'| \le \eps, \quad |a_i| \le \eps^{-2} R^{-1}, $$
is tangent to $\{u=0\}$ at $0$ for some small $\eps$ that satisfies
$\eps \ge R^{-\sigma/2}$, and assume further that all balls of radius $\frac 12 R^{1-\sigma}$ which are tangent to $\Gamma$ by below are included in $\{u<0\}$. Then
$$\sum_1^{n-1} a_i \le CR^{-1}.$$
\end{prop}

Proposition \ref{p2} states that the blow-down of $\{u=0\}$ satisfies the minimal surface equation in some viscosity sense.
Indeed, if we take $\eps=R^{-\sigma/2 }$, then the set $R^{\sigma-1}\{u=0\}$ cannot be touched at $0$ in a $R^{-1/2}$ neighborhood of the origin by a surface with curvatures bounded by $1/2$ and mean curvature greater than $CR^{-\sigma}$.

\begin{proof}

We argue as in the proof of Proposition \ref{p1} except that now we replace $\p B_R(-R e_n)$ by $\Gamma$ and $ \p B_{R/8}(z_0)$ by
$$\Gamma_2:=\left \{x_n= \sum_1^{n-1} \frac{a_i}{2} \, x_i^2  + b' \cdot x' +\frac{C_1}{R}- \frac 1 R |x'|^2\right \}.$$
We claim that
\be\label{104}
U(x,y) < G_{R/8}(d_2(x),y) \quad \mbox{in} \quad \mathcal C_{R}\setminus \mathcal C_{R/2},
\ee
where $d_2$ represents the signed distance to the $\Gamma_2$ surface and $\mathcal C_R$ is defined in \eqref{CR}. 
Using the surfaces $\Psi_{r,z}$ as comparison functions we obtain as in \eqref{urho}, \eqref{ug} above that 
$$U(x,y)<G_{R/8}(d_1(x)+C'r^{-1},y) \quad \mbox{in $\mathcal C_R$},$$
with $d_1(x)$ representing the signed distance to $\Gamma$. Notice that \eqref{d12} is valid in our setting. Now we argue as in \eqref{ug1} and obtain the desired claim  \eqref{104}.

Next we show that $G_{R/8}(d_2(x),y)$ is a supersolution away from the set $\{|d_2| \le 1, y=0\}$ provided that
$$\sum_1^{n-1} a_i \ge M R^{-1},$$
for some $M$ large, universal to be made precise later. 
The boundary inequality on $\{y=0\}$ is clearly satisfied  and on $\{y>0\}$ we have
\be\label{ta}
\triangle_a G_{R/8}(d_2(x),y) = \triangle _a G_{R/8} (s,y) + H(x) \partial_s G_{R/8}(s,y) , \quad \quad s:=d_2(x),
\ee
where $H(x)$ represents the mean curvature at $x$ of the parallel surface to $\Gamma_2$, and $\triangle_a$ on the right hand side is with respect to the variables $(s,y)$. 
If $|s|> R^{1-\delta}$ then $\p_s G_{R/8}=0$, and if $|s|\le R^{1-\delta}$ we show below that $H<0$, and in both cases we obtain $\triangle_a G_{R/8} \le 0$.

Let $\kappa_i$, $i=1,..,n-1,$ be the principal curvatures of $\Gamma_2$ at the projection of $x$ onto $\Gamma_2$. Notice that at this point the slope of the tangent plane to $\Gamma_2$ is less than $4 \eps$ hence we have
$$|\kappa_i| \le 2 \eps^{-2}R^{-1} \le 2R^{\sigma-1}, \quad \sum \kappa_i \le - \sum a_i + C \eps^2 \max |a_i| \le - \frac 12 M R^{-1}. $$ When $ |d_2| \le  R^{1-\delta},$ we obtain $d_2 \kappa_i =o(1)$, $d_2 \kappa_i^2=o(R^{-1})$ (since $\sigma < \delta/3$), hence
\be\label{H}
H(x)=\sum \frac{\kappa_i}{1-d_2 \kappa_i} =\sum (\kappa_i +  \frac{d_2\kappa_i^2}{1-d_2 \kappa_i}) \le - \frac 1 4 M R^{-1}.
\ee

Now we translate the graph of $G_{R/8}(d_2,y)$ along the $e_n$ direction till it touches the graph of $U$ by above. 
Precisely, we consider the graphs of $G_R(d_2(x-te_n),y) $ with $t \le 0$ 
and start with $t$ negative so that the function is identically 1 in $\mathcal C_R$. 
Then we increase $t$ continuously till this graph becomes tangent by above to the graph of $U$ in $\mathcal C_R$. 
Since $u(0)=0$, a contact point must occur for some $t \le 0$ and, by \eqref{104}, 
this point is interior to $\mathcal C_{R/2}$ and lies on $y=0$. 
Let $(x^*,0,u(x^*))$ be the first contact point where a translate $G_{R/8}(d_2(x-t^*e_n),y)$ touches $U$ by above. We show that we reach a contradiction if $M$ is chosen sufficiently large.

Define $V$ as $$ V(x,y):=G (d_2(x-t^*e_n) + C/R,y) \ge G_{R/8}(d_2(x-t^*e_n),y) \ge U(x,y).$$ Notice that
$$  \p_y^{1-a} V = W'(V),  \quad \quad (V-U)(x^*,0) \le C/R. $$
In $\mathcal B_1(x^*)$ we use the computation \eqref{ta} above for $V$ together with \eqref{H} and obtain 
$$\triangle _a V \le - c M R^{-1} , \quad \quad \mbox{in} \quad\mathcal B_1(x^*) .$$ 
The function $Q:=(V-U)/(cMR^{-1}) \ge 0$ satisfies in $\mathcal B_1(x^*) $
$$\triangle_a Q \le -1, \quad |\p_y^{1-a} Q| \le C Q, \quad Q(x^*,0) \le C' M^{-1}.$$
By the maximum principle 
$$ Q(x,y) \ge \mu^2 + \mu \, y^{1-a} - \frac{1}{2(n+1)}(|x-x^*|^2+y^2),$$
for some $\mu$ small universal, and we reach a contradiction at $(x^*,0)$ if $M$ is sufficiently large.

\end{proof}

\section{Harnack inequality}

In this section we use Proposition \ref{p1} and prove a Harnack inequality property for flat level sets, see Theorem
\ref{TH} below. The key step in the proof is to control the $x_n$ coordinate of the level set $\{u=0\}$ in a set of large measure in the $x'$-variables.

{\it Notation:} We denote by $\mathcal C(l, \theta)$ the cylinder
$$\mathcal C(l,\theta):= \{|x'| \le l, \quad |x_n| \le \theta\}.$$

\begin{thm}[Harnack inequality for minimizers]{\label{TH}}

Let $U$ be a minimizer of $J$ in $\mathcal B_{q}$ and assume that
$$0 \in \{u=0\} \cap \mathcal C(l,l) \subset \mathcal C(l,\theta),$$
and that all balls of radius $q:=(l^2 \theta^{-1})^{1- \frac \sigma 2} $ which are tangent to 
$\mathcal C(l,\theta)$ by below and above are included in $\{u<0\}$ respectively $\{u>0\}$.

Given $\theta_0>0$ there exist $\omega>0$ small depending
on $n$, $W$, and $\varepsilon_0(\theta_0) >0$ depending on
$n$, $W$ and $\theta_0$, such that if
$$\theta l^{-1} \le \varepsilon_0(\theta_0), \quad \theta_0 \le \theta,$$
then
$$ \{u=0\} \cap \mathcal C(\bar l, \bar l) \subset \mathcal C( \bar l, \bar \theta), \quad \quad \bar l:= l/4, \quad \bar \theta:=(1-\omega) \theta,$$
and all balls of radius $\bar q:=( {\bar l}^2 {\bar \theta}^{-1})^{1- \frac \sigma 2} $ which are tangent to $\mathcal C(\bar l, \bar \theta)$ by below or above do not intersect $\{u=0\}$.

\end{thm}

The fact that $u$ is a minimizer of $J$ is only used in a final step of the proof. This hypothesis can be replaced by $x_n$ monotonicity for $u$, or more generally by the monotonicity of $u$ in a given direction which is not perpendicular to $e_n$.  

\begin{defn}
For a small $a>0$, we denote by $ \mathcal D_a$ the set of points on $$\{u=0\} \cap \mathcal C(\frac 34 l, \theta) $$ which have a paraboloid of opening $-a$ and vertex $y=(y',y_n)$  
$$P_{a,y}:=\left\{ x_n=-\frac a 2 |x'-y'|^2 + y_n \right \} $$
tangent by below in $\mathcal C(l, \theta)$, and with $P_{a,y}$ below the lateral boundary of 
$\mathcal C(l, \theta)$. In other words we allow only those polynomials $P_{a,y}$ which exit $\mathcal C(l, \theta)$ through the ``bottom".

We denote by $D_a \subset \R^{n-1}$ the projection of $\mathcal D_a$ into $\R^{n-1}$ along the $e_n$ direction. 

\end{defn}

By Proposition \ref{p1} we see that as long as
\be\label{al}
 l^{-1} \ge a \ge l^{-2-\eta}, \quad \quad \mbox{and} \quad l \ge C(\theta _0),
 \ee 
 for some $\eta$ small universal (depending on $\sigma$), then $\{u=0\}$ has the following property $(P)$:
 
 \
 
$(P)$ In a neighborhood of any point of $\mathcal D_a$ the set $\{u=0\}$ is a graph in the $e_n$ direction of a $C^2$ function with second derivatives bounded by $\Lambda a$ with $\Lambda$ a universal constant.   
 
 \
 
Indeed, since $a \le l^{-1}$, at a point $z \in \mathcal D_a$ the corresponding paraboloid at $z$ has a tangent ball of radius 
$$R:=c a^{-1} \le l^{2+\eta}$$ by below. 
Since $|z'| \le 3/4l$ we see that $\{u=0\} \cap B_{l/4}(z)$ has a tangent ball $B_R(x_0)$ by below at $z$ and hypothesis a) of Proposition \ref{p1} holds since
$$l/4 \ge  R^{\frac 1 2-\sigma}.$$ 

 The assumption that all balls of radius $q \ge c(\theta_0) l^{2-\sigma} \ge R^{1-\sigma} $ tangent by below to 
$C(l,\theta)$ are included in $\{u<0\}$ gives that all balls tangent to $\p B_R(x_0) \cap B_{l/4}(z)$ by below are also included in $\{u<0\}$ hence hypothesis b) of Proposition \ref{p1} holds. 

Since $u$ is a minimizer, in any sufficiently large ball in $\{u<0\}$ we have points that satisfy $u<-1+c$ and hypothesis c) holds as well. In conclusion Proposition \ref{p1} applies and property ($P$) holds.

Since $\{u=0\}$ satisfies property $(P)$ then it satisfies a general version of Weak Harnack inequality which we proved in \cite{S2}. 
In particular we are in the setting of Propositions  6.2 and 6.4 (see also Remark 6.7) in \cite{S2}. 

This means that for any $\mu>0$ small, there exists $M(\mu)$ depending on $\mu$ and universal constants such that if
\be\label{201}
\{ u=0 \} \cap \left  (  B'_{l/2} \times [-\theta, (\omega-1) \theta]  \right  )  \ne \emptyset ,  \quad \quad \omega:= (32 M)^{-1},\ee
then, by Proposition 6.2 in \cite{S2}, we obtain 
\be\label{202}
 \mathcal H^{n-1}(D_a \cap B'_{l/2}) \ge (1-\mu) \mathcal H^{n-1}(B'_{l/2}),  \quad \quad \mbox{with} \quad a:= M \, \omega \, \theta l^{-2},  \ee
and
\be\label{203}
\mathcal D_a \cap \{|x'| \le l/2  \} \subset \{x_n \le (8 M \omega - 1)  \theta \} = \{ x_n \le -3\theta /4\} . \ee
We can apply Proposition 6.2 in \cite{S2} since the interval $I$ of allowed openings of the paraboloids satisfies (see \eqref{al}) 
$$I=[\omega \, \theta l^{-2}, M \omega \, \theta l^{-2}] \subset [l^{-2-\eta}, l^{-1}],$$ 
provided that $l \ge C(\mu, \theta_0)$ and $\eps_0 \le c$. 

Next we let $\mathcal D^*_a$ to denote the set of points on 
\be\label{2035}
\mathcal D_a^*:= \{u=0\} \cap \left(\{|x'| \le l/2\}\times [-\frac 12 \theta, \theta] \right) \ee
 which admit a tangent paraboloid of opening $a$ by above which exists $\mathcal C(l,\theta)$ through the ``top". Also we denote by $D_a^*\subset \R^{n-1}$ the projection of $\mathcal D_a^*$ along $e_n$. Then according to Proposition 6.4 in \cite{S2}, (applied ``up-side down") we have
\be\label{204}
 \mathcal H^{n-1}(D^*_{\tilde a} \cap B'_{l/2}) \ge \mu_0 \, \mathcal H^{n-1}(B'_{l/2}), \quad \quad \mbox{with} \quad \tilde a=8 \theta l^{-2},
\ee
for some $\mu_0$ universal.

We choose $\mu$ in \eqref{201}-\eqref{203} universal as $$\mu := \mu_0 /2.$$ According to \eqref{202}, \eqref{204} this gives
\be \label{205} 
\mathcal H^{n-1}(D_a \cap D^*_{\tilde a}) \ge  \frac{\mu_0}{2} \mathcal H^{n-1}(B'_{l/2}).\ee
Notice that by \eqref{203}, \eqref{2035} the sets $\mathcal D_a$ and $\mathcal D^*_{\tilde a}$ are disjoint.

 At this point we would reach a contradiction (to \eqref{201}) if $\{u=0\}$ were assumed to be a graph in the $e_n$ direction. Instead we use \eqref{205} and show that $U$ cannot be a minimizer.

\

{\it Proof of Theorem \ref{TH}.}

It suffices to show that 
$$\{u=0\} \cap \mathcal C(l/2, l/2) \subset C(l/2, (1-\omega)\theta).$$
Then the existence of the balls of size $q \ll l^2 \theta^{-1}$ (included in $\{u<0\}$ and  $\{u>0\}$ respectively) tangent to $C(l/4, (1-\omega)\theta)$ 
follows easily as we restrict from the cylinder of size $l/2$ to the one of size $l/4$, and the conclusion is satisfied since $\tilde q\le q$. 

\

Assume by contradiction that \eqref{201} holds, and therefore \eqref{202}, \eqref{205} hold as well. 
For each $x \in \mathcal D_a$ the set $\{u=0\}$ has a tangent ball of radius $c a^{-1} \ge c l$ by below. Moreover, the normal to this balls at the contact points and the $e_n$ direction make a small angle which is bounded by $c  \, \, \theta l^{-1} \le c \eps_0$.  According to Lemma \ref{l3} part 2) and Remark \ref{r1}, we conclude that for any fixed constant $K$ we have
\be\label{UG}
\max_{(t,y) \in \mathcal B_K^+}|U(x',x_n+t,y)-G(t,y)|  \le \rho, 
\ee
with $\rho=\rho(K, \eps_0) \to 0$ as $\eps_0 \to 0$.

We denote the 2D  half disk of radius $r$ in the $(x_n,y)$-variables centered at $z \in \R^n$ as 
$$\mathcal B^+_{r,z}:=\{  (z',z_n+t,y)|  \quad |(t,y)| \le r, \quad y \ge 0\}.$$
From above we find for all $x \in \mathcal D_a$, or similarly if $x \in \mathcal D_a^*$, we have
\be\label{JU}
J(U, \mathcal B^+_{K,x}) \ge \mathcal J(G, \mathcal B_K^+) - \bar \rho,
\ee
with $\bar \rho=\bar \rho(K, \eps_0) \to 0$ as $\eps_0 \to 0$.

If $x' \in D_a \cap D^*_{\tilde a}$ then by \eqref{203}, \eqref{2035} the two points $x^1=(x',x^1_n) \in \mathcal D_a$ and $x^2=(x',x^2_n) \in \mathcal D^*_{\tilde a}$ satisfy $x^2_n -x_n^1 \ge \theta /4 \ge \theta_0 /4$. By \eqref{UG} this means that the two disks $\mathcal B_{K,x^i}$ are disjoint provided that $\rho$ is small, thus
$$\mathcal J(U, \mathcal B^+_{l/2,(x',0)}) \ge 2\left( \mathcal J(G, \mathcal  B^+_K) - \bar \rho \right) \quad \quad \mbox{if}
\quad x' \in D_a \cap D^*_{\tilde a}.$$ 

We integrate in $x'$ and use also \eqref{202}, \eqref{205}, \eqref{JU} to obtain
$$\mathcal J\left (U, A_{l/2}\right) \ge (1+ \mu_0/2) \left( \mathcal J(G, \mathcal  B^+_K) - \bar \rho \right) \, \mathcal H^{n-1}(B'_{l/2}),$$
with
$$A_{l/2}:=\mathcal C( \frac l2, \frac l 2) \times [0,\frac l2].$$
We choose first $K$ large and then $\eps_0$ small such that $\bar \rho$ is sufficiently small such that
$$\mathcal J\left (U, A_{l/2}\right) \ge (1+ \mu_0/4) \,  \mathcal J(G, \R^2_+) \, \mathcal H^{n-1}(B'_{l/2}).$$

This contradicts Lemma \ref{l1conv} below provided that $\eps_0$ is taken sufficiently small.

\qed

Next lemma is a $\Gamma$-convergence result and it is a consequence of the minimality of $U$ in $A_{l/2}$.

\begin{lem}\label{l1conv}
\begin{equation}{\label{conv2}}
\mathcal J( U , A_{l/2}) \le \mathcal J(G, \R^2_+) \, \,  \mathcal H^{n-1}(B_{l/2}') + \gamma(\eps_0)\,  l^{n-1} .
\end{equation}
with $\gamma(\eps_0) \to 0$ as $\eps_0 \to 0$.
\end{lem}

\begin{proof}
We interpolate between $U$ and $V(x,y):=G(x_n,y)$ as
$$H=(1-\varphi)U + \varphi V,.$$
Here $\varphi $ is a cutoff Lipschitz function such that
$\varphi =0$ outside $A_{l/2}$, $\varphi=1$ in $\mathcal R$ and $|\nabla  \varphi| \le 8/(1+y)$ in $A_{l/2}\setminus \mathcal R$,  where $\mathcal R$ is the cone
$$\mathcal R:= \{(x,y) | \quad |x| \le l/2 -1 - 2y  \}. $$
By minimality of $U$ we have
$$\mathcal J(U,A_{l/2}) \le \mathcal J(H,A_{l/2})= \mathcal J(V, \mathcal R) + \mathcal J(H, A_{l/2} \setminus \mathcal R).$$
Since $$\mathcal J(V, \mathcal R) \le \mathcal J(V, A_{l/2}) \le \mathcal J(G, \R^2_+)\, \mathcal H^{n-1}(B_{l/2}'),$$
we need to show that 
\be\label{JH}
\mathcal J(H, A_{l/2} \setminus \mathcal R) \le \gamma \, l^{n-1}
\ee with $\gamma$ arbitrarily small.
We have
$$\mathcal J(H, A_{l/2} \setminus \mathcal R) \le 4 \int_{A_{l/2} \setminus \mathcal R} \left( |\nabla \varphi|^2(V-U)^2 + |\nabla (V-U)|^2 \right) y^a  \, dx dy $$
\be\label{JHA}
+ \int_{B_{l/2} \setminus B_{l/2-1}} (v-u)^2 dx.
\ee
We use that $|U|,|V| \le 1$, $|\nabla U|, |\nabla V| \le C/(1+y)$ and we see that in \eqref{JHA} the first integral in the region where $y \ge C \gamma ^ {1/a}$ is bounded by
$$\int_{C \gamma^{1/a}}^{l/2} C_1 (1+y)^{-2} (1+y) \, y^a  dy \le \gamma /4.$$
Next we notice that $u$ and $v$ are sufficiently close to each other in $\mathcal C(l/2,l/2)$ away from a thin strip around $x_n=0$. Indeed,
we can use barrier functions as in Proposition \ref{p1} (see \eqref{urho}) and bound $u$ by above an below in terms of the function $\psi_{l/2}$ and distance to the hyperplanes $x_n=\pm \theta$. This implies that 
$$|v-u| \le \gamma \quad \mbox{in $\mathcal C(l/2,l/2)$ if $|x_n| \ge C(\gamma)+\theta$.}$$
with $C(\gamma)$ large, depending on the universal constants and $\gamma$. 
For the extensions $U$ and $V$ this gives
$$|V-U|, |\nabla (V-U)| \le C_2 \gamma \quad \mbox{in $A_{l/2}$ if $|x_n| \ge C'(\gamma)+\theta$ and $y \le C \gamma^{1/a}$,}$$
with $C_2$ universal. Now \eqref{JH} easily follows from \eqref{JHA}. 

\end{proof}

\section{Improvement of flatness}

We state the improvement of flatness property of minimizers.

\begin{thm}[Improvement of flatness]{\label{c1alpha}}

Let $U$ be a minimizer of $J$ in $\mathcal B_{q}$ and assume that
$$0 \in \{u=0\} \cap \mathcal C(l,l) \subset \mathcal C(l,\theta),$$
and that all balls of radius $q:=(l^2 \theta^{-1})^{1- \frac \sigma 2} $ which are tangent to 
$\mathcal C(l,\theta)$ by below and above are included in $\{u<0\}$ respectively $\{u>0\}$.

Given $\theta_0>0$ there exist $\eta>0$ small depending
on $n$, and $\varepsilon_1(\theta_0) >0$ depending on
$n$, $W$ and $\theta_0$, such that if
$$\theta l^{-1} \le \varepsilon_1(\theta_0), \quad \theta_0 \le \theta,$$
then
$$\{u=0 \} \cap   \mathcal C_\xi (\bar l, \bar l) \subset C_\xi(\bar l, \bar \theta), \quad \quad \bar l:=\eta l, \quad \bar \theta:=\eta^{3/2} \theta,$$
 and all balls of radius $\bar q:=( {\bar l}^2 {\bar \theta}^{-1})^{1- \frac \sigma 2} $ which are tangent to $\mathcal C_\xi(\bar l, \bar \theta)$ by below and above are included in $\{u<0\}$ respectively $\{u>0\}$. 

Here $\xi \in \R^n$ is a unit vector and $C_\xi(\bar l, \bar \theta)$ represents the cylinder with axis $\xi$, base $\bar l$ and height $\bar \theta$.  
\end{thm}

As a consequence of this flatness theorem we obtain our main theorem.

\begin{thm}{\label{planelike}}
Let $U$ be a global minimizer of $J$. Suppose that the $0$ level set $\{u=0\}$ is asymptotically flat at $\infty$, i.e there exist sequences of positive numbers
$\theta_k$, $l_k$ and unit vectors $\xi_k$ with $l_k \to \infty$,
$\theta_k l_k^{-1}\to 0$ such that
$$\{u=0\} \cap B_{l_k} \subset \{ |x \cdot \xi_k| < \theta_k \}.$$
Then the $0$ level set is a hyperplane and $u$ is one-dimensional.
\end{thm}

 By saying that $u$ is one-dimensional we understand that $u$ depends only on one direction $\xi$, i.e  $u=g(x \cdot \xi)$.

\begin{proof} Without loss of generality assume $u(0)=0$. Fix $\theta_0>0$, and $\eps \le \eps_1(\theta_0)$. We choose $k$ sufficiently large such that, after increasing $\theta_k$ if necessary we have $\theta_k
l_k^{-1} =  \varepsilon$. We can apply Theorem \ref{c1alpha} since $q=(l_k \eps^{-1})^{1-\frac \sigma 2} \ll l_k$, and we obtain
that $\{ u=0 \}$ is trapped in a flatter cylinder. We apply Theorem
\ref{c1alpha} repeatedly till the
height of the cylinder becomes less than $\theta_0$. We conclude that $\{u=0\}$ is trapped in a cylinder withe flatness less than $\eps$ and height $\theta_0$. We let first $\eps \to 0$ and then $\theta_0 \to 0$ and obtain the desired conclusion.

\end{proof}

\

{\it Proof of Theorem \ref{c1alpha}}

The proof is by compactness and it follows from Theorem \ref{TH} and Proposition \ref{p2}.
Assume by contradiction that there exist $U_k$, $\theta_k$, $l_k$,
$\xi_k$ such that $u_k$ is a minimizer of $J$,
$u_k(0)=0$, and the level
set $\{ u_k=0 \}$ stays in the flat cylinder $\mathcal C(l_k,\theta_k)$
with $\theta_k \ge \theta_0$, $\theta_k l_k^{-1} \to 0$ as $k \to \infty$ for
which the conclusion of Theorem \ref{c1alpha} doesn't hold.\\

Let $A_k$ be the rescaling of the $0$ level sets given by
$$ (x',x_n) \in \{ u_k =0 \} \mapsto (z',z_n) \in A_k$$
$$ z'=x'l_k^{-1}, \quad z_n=x_n \theta_k^{-1}.$$

{\it Claim 1:} $A_k$ has a subsequence that converges uniformly on $|z'|
\le
1/2$ to a set $A_{\infty}=\{(z',w(z')), \quad |z'| \le 1/2 \}$ where $w$
is a Holder continuous function. In other words, given $\varepsilon$, all
but a finite number of the $A_k$'s  from the subsequence are in an
$\varepsilon$ neighborhood of $A_{\infty}$.\\

{\it Proof:} Fix $z_0'$, $|z_0'| \le 1/2$ and suppose $(z_0', z_k) \in
A_k$. We apply Theorem \ref{TH} for the function $u_k$ in the
cylinder
$$ \{ |x'-l_k z_0'| < l_k/2 \} \times \{ |x_n- \theta_k z_k | <
2 \theta_k \}$$
in which the set $\{ u_k =0 \}$ is trapped. Thus, there exist an increasing function $\varepsilon_0(\theta)>0$,
$\varepsilon_0(\theta) \to 0$ as $\theta \to 0$, such that $\{ u_k =0 \}$
is
trapped in the cylinder
$$ \{ |x'-l_kz_0'| <  l_k/8 \} \times \{ |x_n- \theta_k z_k| < 2
(1-\omega) \theta_k \}$$
provided that $4\theta_k l_k^{-1} \le \varepsilon_0(2\theta_k)$. Rescaling back we find that
$$A_k \cap \{|z'-z_0'|\le 1 /8 \} \subset \{ |z_n-z_k| \le 2
(1-\omega) \}.$$
We apply the Harnack inequality repeatedly and we find that
\begin{equation}{\label{c30}}
A_k \cap \{|z'-z_0'|\le 2^{-2m -1} \} \subset \{ |z_n-z_k| \le 2
(1-\omega)^m \}
\end{equation}
provided that
$$ \theta_k l_k^{-1} \le 4^{-m-1} \varepsilon_0 \left (2(1-\omega)^m
\theta_k \right).$$
Since these inequalities are satisfied for all $k$ large we conclude that
(\ref{c30}) holds for all but a finite number of $k$'s.
Now the claim follows from Arzela-Ascoli Theorem.\\

{\it Claim 2:} The function $w$ is harmonic (in the viscosity sense).\\

{\it Proof:} The proof is by contradiction. Fix a quadratic polynomial
$$z_n=P(z')=\frac{1}{2}{z'}^TMz' + \xi \cdot z', \quad \quad \|M \|
< \delta^{-1}, \quad |\xi| <\delta^{-1}$$
such that $ tr \, M >\delta $, $P(z')+ \delta  |z'|^2$ touches
the graph of $w$, say, at $0$ for simplicity, and stays below $w$ in
$|z'|<8\delta$, for some small $\delta$. 
Notice that at all points in the cylinder $|z'|<2\delta$, the quadratic polynomial above admits a tangent paraboloid by below of opening $- \delta^{-2}$ which is below $z_n=-2$ when $|z'| \ge 6 \delta$. 

Thus,
for all $k$ large we find points $({z_k}',{z_k}_n)$ close to $0$ such that
$P(z')+ const$ touches $A_k$ by below at $({z_k}',{z_k}_n)$ and stays
below it in $|z'-{z_k}'|< \delta$.\\
This implies that, after eventually a translation, there exists a surface
$$\Gamma:=\left \{ x_n =
\frac{\theta_k}{l_k^2} \frac{1}{2} {x'}^T M x'+ \frac{\theta_k}{l_k} \xi_k
\cdot x'
\right \}, \quad |\xi_k|<2 \delta^{-1}$$
that touches $\{ u_k = 0 \}$ at the origin and stays below it in $\mathcal C(\delta l_k, 2 \theta_k)$. 
Moreover in the cylinder $\mathcal C(l_k/2, 2 \theta_k)$ the surface $\Gamma$ admits at all points with $|x'| \le \delta l$ a tangent ball by below of radius $\delta^2 l_k^2 \theta_k^{-1} \gg q$. In view of our hypothesis we conclude that $\Gamma \cap B_{\delta l_k}$ admits at all its points a tangent ball of radius $q$ by below which is included in $\{u<0\}$.

We contradict Proposition \ref{p2} by choosing $R$ as
$$ R^{-1}:=C^{-1} \, \delta \, \theta_k l_k^{-2} ,$$ with $C$ the constant from Proposition \ref{p2} and with $\eps = \delta^2 $. 
Then for all large $k$ we have
$$ \theta_k l_k^{-1}|\xi_k| \le \eps, \quad \quad \theta_k l_k^{-2} \|M\| \le \eps^{-2} R^{-1}, \quad \quad  \delta l_k \ge R^{\frac 12 -\sigma}, \quad q \ge R^{1-\sigma},$$
and Proposition \ref{p2} applies. We obtain $tr \, M \le \delta$ and we reached a contradiction.

\

Since $w$ is harmonic, there exist $0 < \eta$ small
depending only on $n$ such that
$$|w-\xi \cdot z'| < \frac 12 \eta^{3/2} \quad \quad \mbox{ for $|z'|<2 \eta$ },$$
and the parabolas of opening $-C$ tangent by below (and above) to
$$z_n= \xi \cdot z' \pm \frac 12 \eta^{3/2}$$
 in the cylinder $|z'|< 2 \eta$ lie below (or above) to the graph of $w$. 

Rescaling back and using the fact that $A_k$ converge uniformly to the
graph of $w$ and that $\bar q < q$ we easily conclude that $u_k$ satisfies the
conclusion of the Theorem \ref{c1alpha} for $k$ large enough, and we reached a contradiction.

\qed

\end{document}